\documentclass[11pt, final]{article}
 \usepackage{a4}
 \usepackage{amsmath}%
 \usepackage{amstext}%
 \usepackage{amssymb}%
 \usepackage{showkeys}%
 \usepackage{cite}

\newtheorem{theorem}{Theorem}

\newtheorem{corollary}[theorem]{Corollary}

\newtheorem{lemma}[theorem]{Lemma}

\newtheorem{proposition}[theorem]{Proposition}
\newcounter{remark}
\newtheorem{rem}[remark]{Remark}
\newenvironment{remark}{\rem\rm}{\endrem}

\newenvironment{proof}[1][Proof]{\noindent\textbf{#1.} }{\ \rule{0.5em}{0.5em}}

\newcommand{\R}{\mathbb{R}}%
\DeclareMathOperator*\dom{dom}%
\DeclareMathOperator*\cl{cl}%
\DeclareMathOperator*\co{co}%
\DeclareMathOperator*\cone{cone}%
\DeclareMathOperator*\inte{int}%
\DeclareMathOperator*\sqri{sqri}%
\DeclareMathOperator*\pr{Pr}%

\title{Approaching the maximal monotonicity of bifunctions via representative functions
\thanks{Research partially supported by DFG (German Research Foundation), project WA 922/1-3.}}

\author{Radu Ioan Bo\c t
\thanks {Faculty of Mathematics, Chemnitz University of Technology,
D-09107 Chemnitz, Germany, e-mail:
bot@mathematik.tu-chemnitz.de.} \and Sorin-Mihai Grad \thanks
{Faculty of Mathematics, Chemnitz University of Technology,
D-09107 Chemnitz, Germany, e-mail:
grad@mathematik.tu-chemnitz.de.}}
 \date{ }
\begin{document}
 \maketitle
\textbf{Abstract.} We provide an approach to maximal monotone bifunctions based on the theory of representative functions. Thus we extend
to nonreflexive Banach spaces recent results due to A.N. Iusem and, respectively, N. Hadjisavvas and H. Khatibzadeh, where sufficient conditions guaranteeing the maximal monotonicity of bifunctions were introduced. New results involving the sum of two monotone bifunctions are also presented.\\



\section{Introduction and preliminaries}

The monotone bifunctions were so far studied by means of equilibrium problems in papers like \cite{HK, ACR, IU}, following the path initiated in the seminal paper \cite{BO}, these contributions establishing different connections they share with the monotone operators. On the other hand, after a flourishing era in the seventies, the study of the monotone operators was resurrected by the rediscovery of the Fitzpatrick function and the introduction of to it related representative functions, which allowed convex analysis to jump aboard in order to unveil new and interesting results especially concerning their maximality, but not only (cf. \cite{PZ, SS, BS, MS, BGW, BGw, BC}).

To the best of our knowledge, the representative functions were until now not invited to bring their contribution to the ongoing investigations on the maximal monotone bifunctions and with this paper we open them the gate, by showing the immense potential they posses in order to deal with the current issues in this research area. More precisely, we attach a representative function to the monotone operator associated to a monotone bifunction in \cite{IU, HK, ACR} and using it we obtain different results involving (maximal) monotone bifunctions in both reflexive and nonreflexive Banach spaces.

We begin by extending to general Banach spaces the statements from \cite{HK, IU} where sufficient conditions that guarantee the maximal monotonicity of a bifunction were proposed. When the space we work with is taken reflexive, the mentioned results from the literature are rediscovered via easier proofs that involve representative functions and convex analysis techniques and do not require renorming arguments as done in the original papers. Then we deal with the sum of two (maximal) monotone bifunctions.

\subsection{Elements of convex analysis}

Let the separated locally convex space $X$ and its continuous dual space $X^*$. By $\langle x^*, x\rangle$ we denote the value of the
linear continuous functional $x^*\in X^*$ at $x\in X$. Moreover, consider the \textit{coupling function} $c:X\times X^*\rightarrow \R$,
$c(x, x^*)=\langle x^*, x\rangle$. Denote the \textit{indicator} function of $U\subseteq X$ by $\delta_U$ and its \textit{interior}, \textit{conical hull} and \textit{closure}  by $\inte(U)$, $\cone(U)$, and $\cl(U)$, respectively. Moreover, if $U$ is convex its \textit{strong quasi relative interior} is
$$\sqri(U)=\big\{x \in U: \cone(U-x) \mbox{ is a closed linear subspace} \big\}.$$

For a function $f:X\rightarrow \overline\R = \R\cup \{ \pm\infty\}$, we denote its
\textit{domain} by $\dom f=\{x\in X: f(x)< +\infty\}$.
We call $f$ \textit{proper} if $f(x)>-\infty$ for all $x\in X$ and
$\dom f\neq\emptyset$. The \textit{conjugate} function of $f$ is
$f^*:X^*\rightarrow \overline\R$, $f^*(x^*)=\sup\big\{\langle
x^*, x\rangle - f(x) : x\in X\big\}$.
For $x\in X$ such that $f(x)\in \R$ we define the
\textit{(convex) subdifferential} of $f$ at $x$ by $\partial f (x)=\{x^*\in
X^*: f(y)-f(x)\geq \langle x^*, y-x\rangle\ \forall y\in X\}$. When $f(x)\notin \R$ we take by convention $\partial f(x)=\emptyset$.

Between a function and its conjugate there is \textit{Young's inequality} $f^*(x^*)+f(x)\geq \langle x^*, x\rangle$ for all $x\in X$ and all $x^*\in X^*$,
fulfilled as equality by a pair $(x, x^*)\in X\times X^*$ if and
only if $x^*\in \partial f(x)$. Denote also by $\overline{\co} f:X\rightarrow \overline\R$ the
largest convex and lower semicontinuous function everywhere less than or equal to $f$, i.e. the
\textit{lower semicontinuous convex hull} of $f$. The function $f$ is called \textit{upper hemicontinuous} if it is upper semicontinuous on line segments.

Let $A$ and $B$ be two nonempty sets. When $f, g:A\times B\rightarrow \overline\R$ are proper, we consider the function
$f\square_2 g:A\times B\rightarrow \overline\R$, $f\square_2 g(a, b)=\inf
\{f(a, c)+g(a, b-c): c\in B\}$. Denote also by $f^\top$ the \textit{transpose} of $f$,
namely the function $f^\top:B\times A\rightarrow \overline \R, f^\top(b,a)=f(a,b)$ for all $(b,a)\in B\times A$. Moreover, we consider the \textit{projection function} $\pr _A:A\times B\rightarrow A$, defined by $\pr_A(a, b)=a$ for all $(a, b)\in A\times B$.

Take further $X$ to be a normed space. Then $X^*$ is its topological dual and $X^{**}$ its topological \textit{bidual}.
We identify $X$ with its image under the canonical injection of $X$ into subspace of $X^{**}$.
Denote the open ball centered in $x\in X$ having the radius $\varepsilon>0$ by ${\cal B}(x, \varepsilon)$.
If $f:X\rightarrow \overline \R$, one can attach to it also the \textit{biconjugate} function $f^{**}:X^{**}\rightarrow\overline\R$ defined by $f^{**}(x^{**})=\sup_{x^*\in X^*} \{\langle x^{**}, x^*\rangle - f^*(x^*)\}$. Note that if $f$ is proper, convex and lower semicontinuous one has the \textit{Fenchel-Moreau} formula $f(x)=f^{**}(x)$ for all $x\in X$.

\subsection{Monotone operators}

Now let us present some basic things on monotone operators come next, following \cite{MS, SS}. Further $X$ is taken to be a nontrivial real Banach space.
A multifunction $T:X \rightrightarrows X^*$ is called a
\textit{monotone operator} provided that for any $x, y \in X$ one
has $\langle y^*-x^*, y-x\rangle \geq 0$ whenever $x^*\in T(x)$
 and $y^*\in T(y)$. The \textit{domain} of $T$ is $D(T)=\{x\in X: T(x)\neq\emptyset\}$, while its
\textit{range} is $R(T)=\cup \{T(x): x\in X\}$.

A monotone operator $T:X
\rightrightarrows X^*$ is called \textit{maximal} when its
\textit{graph} $G(T)=\{(x, x^*)\in X\times X^*: x^*\in T(x)\}$ is not
properly included in the graph of any other monotone operator $S:
X \rightrightarrows X^*$. The subdifferential of a proper, convex and
lower semicontinuous function on $X$ is a typical example of a maximal monotone operator.

To a monotone operator $T:X \rightrightarrows X^*$ one can attach the
 \textit{Fitzpatrick function}
$$\varphi _T:X\times X^*\rightarrow \overline\R,\ \varphi
_T(x, x^*)=\sup\big\{\langle y^*, x\rangle + \langle x^*, y\rangle
- \langle y^*, y\rangle:y^*\in Ty\big\},$$
which is convex and weak-weak$^*$ lower semicontinuous.

The function $\psi_T:=\overline{\co} (c+\delta_{G(T)})$, where the closure is considered in the strong topology, is very well connected to
the Fitzpatrick function. On $X\times X^*$ we
have $\psi_T^{*\top}=\varphi_T$ and, when $X$ is a reflexive
Banach space, one also has $\varphi_T^{*\top}=\psi_T$.

If $T$ is maximal monotone, then $\varphi_T\geq c$ and $G(T)=\{(x,x^*)\in X\times
X^*:\varphi_T(x,x^*)=\langle x^*,x\rangle\}$. These properties of the Fitzpatrick function motivate attaching to monotone operators other functions, as follows.
If $T$ is a monotone operator, a convex  and strong lower semicontinuous function $h_T:X\times
X^*\rightarrow \overline \R$ fulfilling $h_T\geq c$  and $G(T)\subseteq\{(x,x^*)\in X\times
X^*:h_T(x,x^*)= c(x, x^*)\}$ is said to be a \textit{representative function} of $T$.

Note that if $G(T)\neq\emptyset$ (in particular if $T$ is maximal monotone), then every representative function of $T$ is
proper. It follows immediately that $\varphi_T$ and $\psi_T$ are
representative functions of the maximal monotone operator $T$, too. Some properties of maximal monotone operators and representative functions attached to them that we need further in this paper follow.

\begin{lemma}\label{le1} (cf. \cite{BS})  Let $T:X\rightrightarrows X^*$ be a maximal monotone
operator and $h_T$ a representative function of $T$.
Then
\begin{enumerate}\item[(i)]$\varphi_T(x,x^*)\leq
h_T(x,x^*)\leq \psi_T(x,x^*)$ for all $(x,x^*)\in X\times
X^*$;\item[(ii)]the restriction of $h_T^{*\top}$ to
$X\times X^*$ is also a representative function of
$T$;\item[(iii)] $\big\{(x,x^*)\in X\times X^*:h_T(x,x^*)=c(x, x^*)\big\}=\big\{(x,x^*)\in X\times X^*:h_T^{*\top}(x,x^*)=c(x, x^*)\big\}=G(T)$.\end{enumerate}
\end{lemma}

Now let us give two maximality criteria for monotone operators.

\begin{theorem}\label{th2}
(cf. \cite[Proposition 2.1]{PZ}) Let $X$ be reflexive. If $h: X\times X^* \rightarrow \overline \R$ is a proper, convex and lower semicontinuous function with $h\geq c$, then the monotone operator $\{(x,x^*)\in X\times X^*:h(x,x^*)=c(x, x^*)\}$ is maximal if and only if $h^{*\top}\geq c$.
\end{theorem}

The following statement was formulated as \cite[Theorem 3.1]{MS} with the condition $0\in \sqri \big(\pr_X$ $(\dom h)\big)$. However, by translation arguments one can reformulate it
by considering the more general condition $\sqri \big(\pr_X(\dom h)\big)\neq \emptyset$, as follows.

\begin{theorem}\label{th3} Let $h: X\times X^* \rightarrow \overline \R$ be a proper and convex function with $h\geq c$ and $h^{*\top}\geq c$ on $X\times X^*$. If $\sqri \big(\pr_X(\dom h)\big)\neq \emptyset$, then the operator $\{(x,x^*)\in X\times X^*:h^*(x^*, x)=c(x, x^*)\}$ is maximal monotone.
\end{theorem}

\subsection{Monotone bifunctions}

Let us also present some preliminaries on bifunctions, following \cite{HK, IU}. Let the nonempty set $C\subseteq X$. A function $F:C\times C\rightarrow \R$ is called \textit{bifunction}. The bifunction $F$ is called \textit{monotone} if $F(x, y) + F(y, x)\leq 0$ for all $x, y\in C$. To the bifunction $F$ one can attach the operators $A^F:X\rightrightarrows X^*$ and ${^F}\!\!A:X\rightrightarrows X^*$ defined by
$$A^F(x)=\left\{
\begin{array}{ll}
\{x^*\in X^*: F(x, y)-F(x, x)\geq \langle x^*, y-x\rangle\ \forall y\in C\},& \mbox{ if } x\in C,\\
\emptyset,& \mbox{ otherwise},
\end{array}\right.$$
and, respectively,
$${^F}\!\!A(x)=\left\{
\begin{array}{ll}
\{x^*\in X^*: F(x, x)-F(y, x)\geq \langle x^*, y-x\rangle\ \forall y\in C\},& \mbox{ if } x\in C,\\
\emptyset,& \mbox{ otherwise},
\end{array}\right.
$$
which are monotone when $F(x, x)=0$ for all $x\in C$ and $F$, respectively $-F$, is monotone. Actually, when $F$ is monotone one has $G\big(A^F\big)\subseteq G\big({^F}\!\!A\big)$.

The monotone bifunction $F$ is said to be \textit{maximal monotone} if $A^F$ is maximal monotone and, respectively, \textit{$BO$-maximal monotone} (where $BO$ stands for Blum-Oettli, as this type of monotone bifunction was introduced in \cite{BO}) when for every $(x, x^*)\in C\times X^*$ it holds
$$F(y, x)+\langle x^*, y-x\rangle\leq 0\ \forall y\in C\ \Rightarrow  F(x, y)\geq \langle x^*, y-x\rangle\ \forall y\in C.$$
When $F$ is monotone, note that its $BO$-maximal monotonicity is equivalent to ${^F}\!\!A= A^F$. Thus any maximal monotone bifunction is $BO$-maximal monotone, but the opposite implication is not always valid, as the situation in \cite[Example 2.2]{HK} shows.

A bifunction $F$ is called \textit{locally bounded at a point $\bar x\in X$} if there exists an $\varepsilon >0$ and a $k\in \R$ such that $F(x, y)\leq k$ for all $x, y\in C\cap B(\bar x, \varepsilon)$. In order not to overcomplicate the paper, when $x\in C$ we denote by a slight abuse of notation by $F(x, \cdot) + \delta_C$ the function defined on $X$ with extended real values which is equal to $F(x, \cdot)$ on $C$ and takes the value $+\infty$ otherwise. Then, $A^F(x)=\partial (F(x, \cdot) + \delta_C)(x)$ for all $x\in X$.

We close the section by presenting a statement which holds in a more general framework than originally considered in \cite[Lemma 3]{BO}.

\begin{lemma}\label{le4} Let $F$ and $G$ be two bifunctions defined on the nonempty, convex and closed set $C\subseteq X$, satisfying $F(x, x)=G(x, x)=0$ for all $x\in C$, such that $F$ is monotone, $F(x, \cdot)$ and $G(x, \cdot)$ are convex for all $x\in C$ and $F(\cdot, y)$ is upper hemicontinuous for all $y\in C$. Then the following statements are equivalent
\begin{enumerate}
\item[(i)] $\bar x\in C$ and $F(y, \bar x)\leq G(\bar x, y)$ for all $y\in C$;
\item[(ii)] $\bar x\in C$ and $0\leq F(\bar x, y)+ G(\bar x, y)$  for all $y\in C$.
\end{enumerate}
\end{lemma}

Note that the monotonicity of $F$ is required only for proving the implication ``$(ii)\Rightarrow(i)$'', which actually holds even if the convexity and topological hypotheses are removed. Using Lemma \ref{le4} we prove another statement which will be useful later.

\begin{lemma}\label{le0} Let $F$ be a bifunction defined on the nonempty, convex and closed set $C\subseteq X$, satisfying $F(x, x)=0$ for all $x\in C$. If $F(x, \cdot)$ is convex for all $x\in C$ and $F(\cdot, y)$ is upper hemicontinuous for all $y\in C$, then $G\big({^F}\!\!A\big)\subseteq G\big(A^F\big)$.
\end{lemma}

\begin{proof} Let $(x, x^*)\in G\big({^F}\!\!A\big)$. Then
$x\in C$ and $F(y, x)\leq \langle x^*, x-y\rangle$ for all $y\in C$. By Lemma
\ref{le4}$(i)\Rightarrow(ii)$ one gets $0\leq F(x, y) + \langle x^*, x-y\rangle$ for all $y\in C$, thus $(x, x^*)\in G\big(A^F\big)$.
\end{proof}

\section{Maximality of monotone bifunctions}

In this section let $X$ be, unless otherwise stated, a nontrivial real Banach space, and let $C\subseteq X$ be a nonempty subset of it. Consider the bifunction $F:C\times C\rightarrow \R$ satisfying $F(x, x)=0$ for all $x\in C$. In order to deal with its maximal monotonicity, we attach to $F$ the following functions
$$h_F:X\times X^*\rightarrow \overline\R,\ h_F(x, x^*)= \sup_{y\in C} \big\{\langle x^*, y\rangle - F(x, y)\big\} + \delta_C(x) = (F(x, \cdot)+\delta_C)^*(x^*)+ \delta_C(x)$$
and
$$g_F:X\times X^*\rightarrow \overline\R,\ g_F(x, x^*)= \sup_{y\in C} \big\{\langle x^*, y\rangle + F(y, x)\big\} + \delta_C(x)= (-F(\cdot, x)+\delta_C)^*(x^*)+ \delta_C(x).$$
Note that for $(x, x^*)\in X\times X^*$ one has
$$h_F^*(x^*,x) = \sup_{y\in C}\big\{\langle x^*, y\rangle + (F(y, \cdot)+\delta_C)^{**}(x)\big\}$$ and
$$g_F^*(x^*,x) = \sup_{y\in C}\big\{\langle x^*, y\rangle + (-F(\cdot, y)+\delta_C)^{**}(x)\big\}.$$
Other properties of these functions are given in the following statement, whose proof is trivial.

\begin{lemma}\label{le5} For all $(x, x^*)\in X\times X^*$ it holds $g_F(x, x^*)\geq h_F^*(x^*, x)$, $h_F(x, x^*)\geq c(x, x^*)$ and $ g_F(x, x^*)\geq c(x, x^*)$. If $F$ is monotone, then $h_F(x, x^*)\geq  g_F(x, x^*)$ for all $(x, x^*)\in X\times X^*$.
\end{lemma}

By construction, one has that $h_F(x, x^*)=c(x, x^*)$ if and only if $(x, x^*)\in G\big(A^F\big)$ and $g_F(x, x^*)=c(x, x^*)$ if and only if $(x, x^*)\in G\big({^F}\!\!A\big)$.
However, $g_F$ and $h_F$ are in general neither convex nor lower semicontinuous, therefore they are not always representative functions for $A^F$ in case this is monotone.
Now we are ready to give our main statements, where sufficient conditions for the maximal monotonicity of $A^F$ are provided. We begin with a theorem where $F$ is not even asked to be monotone.

\begin{theorem}\label{th6}  Let $C$ be convex and closed with $\sqri (C)\neq \emptyset$, $F(x, \cdot)$ convex and lower semicontinuous for all $x\in C$ and $F(\cdot, y)$ concave upper semicontinuous for all $y\in C$. Then $A^F$ is maximal monotone and $A^F={^F}\!\!A$.
\end{theorem}

\begin{proof} The convexity and topological assumptions on $C$ and $F(x, \cdot)$, for  $x\in C$, yield that the function $F(x, \cdot)+\delta_C$ is proper, convex and lower semicontinuous whenever $x\in C$. Then $(F(x, \cdot)+\delta_C)^{**}(z)=F(x, z)+\delta_C(z)$ whenever $x\in C$ and $z\in X$, consequently, via Lemma \ref{le5}, $h_F^{*^\top}=g_F\geq c$ on $X\times X^*$. Analogously, the convexity and topological assumptions on $C$ and  $-F(\cdot, y)$, $y\in C$, imply $h_F=g_F^{*^\top}\geq c$ on $X\times X^*$. Obviously, $h_F$ and $g_F$ are in this case convex functions, whose properness follows immediately, too.

One gets $\pr_X(\dom h_F) \subseteq \pr_X(\dom g_F) \subseteq C$. Taking an $x\in C$, since
$F(x, \cdot)+\delta_C$ is proper, convex and lower semicontinuous, its conjugate is proper (cf. \cite[Theorem 2.3.3]{ZZ}), so there exists an $x^*\in X^*$ such that $(F(x, \cdot)+\delta_C)^*(x^*)<+\infty$. Consequently, $h_F(x, x^*)<+\infty$, i.e. $C\subseteq \pr_X(\dom h_F)$. Therefore $\pr_X(\dom h_F) = \pr_X(\dom g_F) = C$. We are now ready to apply Theorem \ref{th3} for $h_F$ and $g_F$, obtaining that the operators (identified through their graphs)
$$\{(x,x^*)\in X\times X^*:h_F^*(x^*, x)=c(x, x^*)\}=\{(x,x^*)\in X\times X^*:g_F(x,x^*)=c(x, x^*)\} = G\big({^F}\!\!A\big)$$ and
$$\{(x,x^*)\in X\times X^*:g_F^*(x^*, x)=c(x, x^*)\}=\{(x,x^*)\in X\times X^*:h_F(x,x^*)=c(x, x^*)\}= G\big(A^F\big)$$ are maximal monotone.

Using Lemma \ref{le0}, it follows $G\big({^F}\!\!A\big)\subseteq G\big(A^F\big)$, consequently, $A^F={^F}\!\!A$, since both are maximal monotone operators.\end{proof}\\

If $C=X$ the condition $\sqri (C)\neq \emptyset$ is automatically satisfied and Theorem \ref{th6} yields the following statement, noting that the lower/ upper semicontinuity of a real valued convex/concave function on the entire space is equivalent to its continuity (cf. \cite[Proposition 2.1.6]{ZZ}).

\begin{theorem}\label{th7}  Let $F(x, x)=0$ for all $x\in X$, $F(x, \cdot)$ be convex and continuous for all $x\in X$ and $F(\cdot, y)$ concave and continuous for all $y\in X$. Then $A^F$ is maximal monotone and $A^F={^F}\!\!A$.
\end{theorem}

\begin{remark}\label{re1} By Theorem \ref{th7} we prove one of the conjectures formulated at the end of \cite{IU}, actually slightly weakening its hypotheses since instead of taking $F$ continuous we ask it to be continuous in each of its variables. If $X$ is reflexive, Theorem \ref{th7} rediscovers \cite[Theorem 3.6(i)]{IU}, bringing the mentioned improvement to its hypotheses.\end{remark}

Taking $F$ to be monotone, here are some hypotheses that guarantee its maximality even in the absence of convexity assumptions in its first variable.

\begin{theorem}\label{th8} Let $C$ be convex and closed and $F$ monotone. If $\sqri (C)\neq \emptyset$, $F(x, \cdot)$ is convex  and lower semicontinuous for all $x\in C$ and $F(\cdot, y)$ upper hemicontinuous for all $y\in C$, then $F$ is maximal monotone.
\end{theorem}

\begin{proof} The convexity and topological assumptions on $C$ and $F(x, \cdot)$, for  $x\in C$, yield that the function $F(x, \cdot)+\delta_C$ is proper, convex and lower semicontinuous whenever $x\in C$. Then $(F(x, \cdot)+\delta_C)^{**}(z)=F(x, z)+\delta_C(z)$ whenever $x\in C$ and $z\in X$, whence $h_F^*(x^*, x)=g_F(x, x^*)$ for all $(x, x^*)\in X\times X^*$. Consequently, via Lemma \ref{le5} and taking into consideration the properties of the conjugate function, one has
\begin{equation}\label{E1}
h_F(x, x^*)\geq \overline{\co} h_F(x, x^*)\geq h_F^*(x^*, x)\geq c(x, x^*)\ \forall (x, x^*)\in X\times X^*.
\end{equation}
Assuming that $h_F$ were improper leads to a contradiction with \eqref{E1}, consequently $h_F$, $\overline{\co} h_F$ and $h_F^*$ are all proper.
Like in the proof of Theorem \ref{th6} one can show that $\pr_X(\dom h_F)=C$. One has
\begin{equation}\label{E5}
{\pr}_X(\dom h_F) \subseteq {\pr}_X(\dom \overline{\co} h_F) \subseteq \overline{\co} {\pr}_X(\dom h_F)
\end{equation}
and, since $C$ is convex and closed, we get $\pr_X\big(\dom \big(\overline{\co} h_F\big)\big) = C$.

In the following we show that
\begin{eqnarray} \label{E7}
G\big(A^F\big) &=& \big\{(x, x^*)\in X\times X^*: \overline{\co} h_F (x, x^*)=c(x, x^*)\big\}\nonumber \\
& = &\big\{(x, x^*)\in X\times X^*: h_F^* (x^*, x)=c(x, x^*)\big\}.
\end{eqnarray}
If $(x, x^*)\in G\big(A^F\big)$, \eqref{E1} yields $h_F^*(x^*, x)=c(x, x^*)$.

Let now $(x, x^*)\in X\times X^*$ for which $h_F^*(x^*, x)=c(x, x^*)$. Then $(x, x^*)\in G\big({^F}\!\!A\big)$, so Lemma \ref{le0} yields $(x, x^*)\in G\big(A^F\big)$. This implies that $\overline{\co} h_F(x, x^*)=c(x, x^*)$ holds if and only if $(x, x^*)\in G\big(A^F\big)$. Applying Theorem \ref{th3} for $\overline{\co} h_F$, it follows that $A^F$ is maximal monotone, i.e. $F$ is maximal monotone, too.\end{proof}\\

When the space $X$ is reflexive, the regularity condition $\sqri (C)\neq \emptyset$ is no longer necessary and we rediscover \cite[Proposition 3.1]{HK}, by means of representative functions, employing tools of convex analysis and without renorming $X$.

\begin{theorem}\label{th9}  Let $X$ be reflexive, $C$ be convex and closed and $F$ monotone. If $F(x, \cdot)$ is convex and lower semicontinuous for all $x\in C$ and $F(\cdot, y)$ upper hemicontinuous for all $y\in C$, then $F$ is maximal monotone.
\end{theorem}

\begin{proof} Things work in the lines of the proof of Theorem \ref{th8}, noticing that \eqref{E1} and \eqref{E7} are fulfilled.
Then we apply Theorem \ref{th3}.\end{proof}\\

When $C=X$ we obtain from Theorem \ref{th8} the following statement.

\begin{theorem}\label{th10} Let $F$ be monotone and fulfilling $F(x, x)=0$ for all $x\in X$. If $F(x, \cdot)$ is convex  and continuous for all $x\in X$ and $F(\cdot, y)$ upper hemicontinuous for all $y\in X$, then $F$ is maximal monotone.
\end{theorem}

\begin{remark}\label{re2} In \cite[Theorem 3.6(ii)]{IU} the same conclusion is obtained when $X$ is reflexive for a monotone bifunction $F$ that fulfills $F(x, x)=0$ for all $x\in X$,
by assuming $F(x, \cdot)$ only convex for all $x\in X$ and $F(\cdot, y)$ continuous for all $y\in X$. However, we doubt that this result holds without any topological assumption on the functions $F(x, \cdot)$, $x\in X$, since in its proof is used \cite[Theorem 3.4(ii)]{IU} whose hypotheses should contain also the lower semicontinuity of $F(x, \cdot)$ for all $x\in X$. A similar comment can be made also for \cite[Theorem 3.6(iii)]{IU} and for the conjectures extending the two mentioned statements to nonreflexive spaces given at the end of \cite{IU}.\end{remark}

Since when we apply Lemma \ref{le4} in the proofs of Theorem \ref{th8} and Theorem \ref{th9} it follows as a byproduct that $F$ is $BO$-maximal monotone, we can rediscover, in the reflexive case, and extend, when $X$ is a general Banach space, \cite[Proposition 3.2]{ACR}, as follows.

\begin{corollary}\label{co11} Let $C$ be convex and closed with $\sqri (C)\neq \emptyset$ and $F$ be $BO$-maximal monotone. If $F(x, \cdot)$ is convex and lower semicontinuous for all $x\in C$, then $F$ is maximal monotone.\end{corollary}

\begin{corollary}\label{co12} Let $X$ be reflexive, $C$ convex and closed and $F$ $BO$-maximal monotone. If $F(x, \cdot)$ is convex and lower semicontinuous for all $x\in C$, then $F$ is maximal monotone.\end{corollary}

When $C=X$ one can formulate another maximality criterium for a monotone bifunction, extending \cite[Proposition 3.5]{HK} to general Banach spaces.

\begin{theorem}\label{th11} Let $F$ be monotone and fulfilling $F(x, x)=0$ for all $x\in X$. If $D\big(A^F\big)=X$ and $F(\cdot, y)$ is upper hemicontinuous for all $y\in X$, then $F$ is maximal monotone.
\end{theorem}

\begin{proof} As $D\big(A^F\big)=X$, for all $x\in X$ one has $\partial F(x, \cdot)(x)\neq \emptyset$, which yields $\overline{\co} F(x, \cdot)(x) = F(x, x) =0$.
On the other hand, for all $x\in X$ it holds $X = \dom F(x, \cdot)\subseteq \dom \overline{\co} F(x, \cdot)$, which implies
$\dom \overline{\co} F(x, \cdot) = X$ and via \cite[Proposition 2.2.5]{ZZ}, as $\overline{\co} F(x, \cdot)(x) =0$, also the properness of $\overline{\co} F(x, \cdot)$.
Then, for any $(x, x^*)\in X\times X^*$, one has
$$
h_F^*(x^*, x)  = \sup_{y\in X} \big\{\langle x^*, y\rangle + (F(y, \cdot))^{**}(x)\big\}= \sup_{y\in X} \big\{\langle x^*, y\rangle + \overline{\co} F(y, \cdot)(x)\big\} \geq$$ 
$$\langle x^*, x\rangle + \overline{\co} F(x, \cdot)(x) =  \langle x^*, x\rangle,$$
consequently, $h_F\geq \overline{\co} h_F\geq h_F^{*\tau}\geq c$ on $X\times X^*$. As $D\big(A^F\big)=X$, $\pr_X (\dom h_F) =X$, using \eqref{E5} it follows
$\pr_X (\dom \overline{\co} h_F)=X$. Applying Theorem \ref{th3} for $\overline{\co} h_F$, the operator having the graph
$\big\{(x, x^*)\in X\times X^*: h_F^* (x^*, x)=c(x, x^*)\big\}$
turns out to be maximal monotone. This graph includes $G\big(A^F\big)$.
To show that the opposite inclusion holds, too, let $(x, x^*)\in X\times X^*$ for which $h_F^* (x^*, x) = c(x, x^*)$. Then
$h_F^* (x^*, x) \leq c(x, x^*)$, so for all $y\in X$ it holds $\overline{\co} F (y, \cdot)(x) \leq \langle x^*, x-y\rangle$. This means nothing but $(x, x^*)\in G\big({^H}\!\!A\big)$, where the bifunction $H:X\times X\rightarrow \R$ is defined by $H(x, y) := \overline{\co} F (x, \cdot)(y)$. It follows immediately that $H(z, z)=0$ for all $z\in X$. As $H(z, \cdot) = \overline{\co} F (z, \cdot)$ is convex for all $z\in X$ and for all $y\in X$ one can verify that $H(\cdot, y)$ is upper hemicontinuous, Lemma \ref{le0} yields $(x, x^*)\in G\big(A^H\big)$.
This means that for all $y\in X$ one has $\overline{\co} F (x, \cdot)(y) \geq \langle x^*, y-x\rangle$, followed by $F (x, y) \geq \langle x^*, y-x\rangle$. Thus $(x, x^*)\in G\big(A^F\big)$, therefore \eqref{E7} holds. Consequently, $F$ is maximal monotone.\end{proof}

\begin{remark}\label{re3} One can see in the proofs of Theorem \ref{th6}, Theorem \ref{th8} and Theorem \ref{th11} that not only $\overline{\co} h_F$ (which coincides with $h_F$ under the hypotheses of the first of them), but also the restriction to $X\times X^*$ of $h_F^{*\top}$ are representative functions of the maximal monotone operator $A^F$.\end{remark}

In Theorem \ref{th6}, Theorem \ref{th8} and Theorem \ref{th11} we have shown with the help of the theory of representative functions that under some hypotheses $A^F$ is maximal monotone. Now let us show that the representative functions of it identified there are actually representative to $A^F$ whenever it is maximal monotone.

\begin{theorem}\label{th13} Let $F$ be maximal monotone. Then $\overline{\co} h_F$ and the restriction to $X\times X^*$ of $h_F^{*\top}$ are representative functions of it.
\end{theorem}

\begin{proof}  The maximal monotonicity of $F$ implies via Lemma \ref{le1} that
\begin{eqnarray*}
G\big(A^F\big) & = & \big\{(x, x^*)\in X\times X^*: \psi_{A^F}(x, x^*)=c(x, x^*)\big\}\\
& = & \big\{(x, x^*)\in X\times X^*: \varphi_{A^F}(x, x^*)=c(x, x^*)\big\}.\end{eqnarray*}
On the other hand, the way $h_F$ is constructed implies
$(c+\delta_{A^F})(x, x^*)\geq h_F(x, x^*)$ for all $(x, x^*)\in X\times X^*$, which yields
$$h_F^*(x^*, x)\geq (c+\delta_{A^F})^*(x^*, x)= \psi_{A^F}^*(x^*, x) = \varphi_{A^F}(x, x^*)\ \forall (x, x^*)\in X\times X^*.$$

Since the monotonicity of $F$ implies, via Lemma \ref{le5}, $h_F(x, x^*)\geq \overline{\co} h_F(x, x^*)\geq h_F^*(x^*, x)$ for all $(x, x^*)\in X\times X^*$, it follows immediately that
for all $(x, x^*)\in X\times X^*$ it holds
$$
\psi_{A^F}(x, x^*)\geq \overline{\co} h_F(x, x^*)\geq h_F^*(x^*, x)\geq \varphi_{A^F}(x, x^*) \geq c(x, x^*).
$$
Consequently,
\begin{eqnarray*}
G\big(A^F\big) &=& \big\{(x, x^*)\in X\times X^*: \overline{\co} h_F(x, x^*)=c(x, x^*)\big\}\\
 & = & \big\{(x, x^*)\in X\times X^*: h_F^*(x^*, x)=c(x, x^*)\big\},
\end{eqnarray*}
which implies that $\overline{\co} h_F$ and $h_F^{*\top}$ restricted to $X\times X^*$ are representative functions of $A^F$.\end{proof}

\begin{remark}\label{re0} In the lines of the proof of Theorem \ref{th13}, one can show that if $T:X\rightrightarrows X^*$ is a maximal monotone operator and $h:X\times X^*\rightarrow \overline \R$ is a function fulfilling $h(x, x^*) \geq h^*(x^*, x)$ for all $(x, x^*)\in X\times X^*$ and $h(x, x^*)\leq c(x, x^*)$ whenever $(x, x^*)\in G(T)$, then
 $\overline{\co} h_F$ and the restriction to $X\times X^*$ of $h_F^{*\top}$ are representative functions of $T$.\end{remark}

\section{The sum of two monotone bifunctions}

One of the most dealt with questions regarding maximal monotone operators is what guarantees that the sum of two of them remains maximal monotone. This issue was extended for maximal monotone bifunctions in \cite{HK}, by means of equilibrium problems. We provide another answer in this matter, preceded by a preliminary result. Even though statements in general Banach spaces can be given (see, for instance, \cite[Chapter VII]{SS}), we take $X$ reflexive in this section. Let $F$ and $G$ be monotone bifunctions defined on $C$.

\begin{lemma}\label{le14} One has $A^F(x)+A^G(x)\subseteq A^{F+G}(x)$ for all $x\in X$ and $F+G$ is monotone.
\end{lemma}

\begin{proof} Let $x\in X$, $y^*\in A^F(x)$ and $z^*\in A^G(x)$. Then $x\in C$ and for all $y\in C$ one has
$F(x, y)\geq \langle y^*, y-x\rangle$ and $G(x, y)\geq \langle z^*, y-x\rangle$. Adding these inequalities, one gets $F(x, y)+G(x, y)\geq \langle y^*+z^*, y-x\rangle$ for all $y\in C$, i.e.
$y^*+z^*\in A^{F+G}(x)$.

Analogously, writing what the monotonicity of $F$ and $G$ means and adding the obtained inequalities one gets that $F+G$ is monotone.\end{proof}

\begin{theorem}\label{th15} Let $F$ and $G$ be maximal monotone, with $f_F$ and $f_G$ two corresponding representative functions. If $0\in \sqri \big(D\big(A^F\big)-D(A^G)\big)$ (or, equivalently, $0\in \sqri \big(\pr_X (\dom f_F) - \pr_X (\dom f_G)\big)$), then $F+G$ is maximal monotone, $A^F+A^G=A^{F+G}$ and $f_F\square_2 f_G$ is a representative function of $A^{F+G}$.
\end{theorem}

\begin{proof} By \cite[Corollary 3.6]{PZ} we obtain that the hypotheses yield the maximal monotonicity of $A^F+A^G$, to which $f_F\square_2 f_G$ is a representative function.
Then Lemma \ref{le14} implies that $A^F(x)+A^G(x)= A^{F+G}(x)$ for all $x\in X$. Consequently, $F+G$ is maximal monotone and $f_F\square_2 f_G$ is a representative function of $A^{F+G}$, too.\end{proof}

\begin{remark}\label{re4} Note that under the hypotheses of Theorem \ref{th15} also the function $(f_F\square_2 f_G)^{*\top}$ is a representative function of $A^{F+G}$. If one takes $f_F:=
\overline{\co} h_F$ and $f_G:=\overline{\co} h_G$, then it holds $$(f_F\square_2 f_G)^*(x^*, x) = \sup_{y\in C}\big\{\langle x^*, y\rangle + (F(y,\cdot)+\delta_C)^{**}(x) + (G(y,\cdot)+\delta_C)^{**}(x)\big\} \leq h_{F+G}^*(x^*,x)$$ for all $(x, x^*)\in X\times X^*$. Thus the just identified representative function of $A^{F+G}$ is smaller than the ones obtained for it via Theorem \ref{th13}.\end{remark}

\begin{remark}\label{re5} If both $F$ and $G$ satisfy the hypotheses of one of Theorem \ref{th6}, Theorem \ref{th8}, Theorem \ref{th9} or, when $C=X$, Theorem \ref{th11}, then $F+G$ fulfills them, too, and this has as consequence its maximal monotonicity.\end{remark}

Now let us present a situation, different from the one displayed in Theorem \ref{th15}, when the inclusion proven in Lemma \ref{le14} turns out to be actually an equality. Note that the reflexivity of the space $X$ plays no role in this statement.

\begin{proposition}\label{pr16} Let $F$ and $G$ be monotone bifunctions defined on the convex and closed set $C$ fulfilling $F(x, x)=G(x, x)=0$ for all $x\in C$, such that for all $x\in C$ the functions $F(x, \cdot)$ and $G(x, \cdot)$ are convex and lower semicontinuous. If $0\in \sqri (C-C)$, then $A^F+A^G= A^{F+G}$.\end{proposition}

\begin{proof} Let $x\in C$. One has $\dom (F(x, \cdot) + \delta_C) = \dom (G(x, \cdot)+\delta_C)= \dom ((F+G)(x, \cdot)+\delta_C)=C$. By definition, $A^F(x)=\partial (F(x, \cdot)+\delta_C)(x)$. Note also that $(F(x, \cdot)+\delta_C)+(G(x, \cdot)+\delta_C)=(F+G)(x, \cdot)+\delta_C$.
By \cite[Theorem 2.8.7]{ZZ}, the hypotheses imply
$$\partial (F(x, \cdot)+\delta_C)(x)+\partial (G(x, \cdot)+\delta_C)(x)=\partial (F(x, \cdot)+G(x, \cdot)+\delta_C)(x).$$ Consequently, $A^F(x)+A^G(x)= A^{F+G}(x)$ and since $x\in C$ was arbitrarily chosen, the conclusion follows.\end{proof}

\begin{remark}\label{re6} Note that the hypotheses of Proposition \ref{pr16} ensure that $\overline{\co} h_{F+G}(x,$ $x^*)\geq h_{F+G}^*(x^*$, $x)\geq c(x, x^*)$ for all $(x, x^*)\in X\times X^*$. Unfortunately, this is not enough in order to guarantee the maximality of $F+G$, which would follow for instance provided the $BO$-maximal monotonicity of this bifunction. However, checking also Remark \ref{re5}, this additional assumption would make, at least in the reflexive case, the condition $0\in \sqri (C-C)$ redundant. Therefore, it remains as an open question what should one add to the hypotheses of Proposition \ref{pr16} in order to obtain the maximality of $F+G$ under no stronger hypotheses than the ones in Remark \ref{re5}.\end{remark}

\noindent{\bf Acknowledgements.} The authors are grateful to their colleague E.R. Csetnek for valuable comments and suggestions on an earlier draft of the paper.

\end{document}